\documentclass[12pt,letterpaper,notitlepage]{amsart}

\usepackage{fullpage}

\usepackage[all]{xy}

\usepackage{verbatim}

\usepackage{amssymb}
\usepackage{amsmath}
\usepackage{amsthm}
\usepackage{amsfonts}

\usepackage{float}

\usepackage[small]{caption}
\usepackage[lofdepth,lotdepth,captionskip=10pt]{subfig}

\usepackage{pstricks}
\usepackage{epsfig}
\usepackage{pst-node}
\usepackage[dvips={-h tir_____.pfb}]{auto-pst-pdf}

\usepackage{enumitem}

\usepackage[linktocpage=true]{hyperref}

\usepackage[alphabetic,msc-links,nobysame,lite]{amsrefs}

\newcommand{\inv}{^{-1}}

\newcommand{\bbC}{\mathbb{C}}
\newcommand{\bbD}{\mathbb{D}}

\newcommand{\bbH}{\mathbb{H}}

\newcommand{\bbP}{\mathbb{P}}
\newcommand{\bbQ}{\mathbb{Q}}
\newcommand{\bbR}{\mathbb{R}}
\newcommand{\bbS}{\mathbb{S}}
\newcommand{\bbT}{\mathbb{T}}

\newcommand{\bbV}{\mathbb{V}}

\newcommand{\bbZ}{\mathbb{Z}}

\newcommand{\calD}{\mathcal{D}}

\newcommand{\calP}{\mathcal{P}}
\newcommand{\calQ}{\mathcal{Q}}

\theoremstyle{plain}

\newtheorem{lemma}{Lemma}[section]
\newtheorem*{lemma*}{Lemma}

\newtheorem*{claim*}{Claim}

\newtheorem*{conjecture*}{Conjecture}

\newtheorem*{corollary*}{Corollary}

\newtheorem*{fact*}{Fact}

\newtheorem*{facts*}{Facts}

\newtheorem*{observation*}{Observation}
\newtheorem{proposition}[lemma]{Proposition}
\newtheorem*{proposition*}{Proposition}

\newtheorem*{question*}{Question}

\newtheorem*{theorem*}{Theorem}

\theoremstyle{definition}

\newtheorem*{definition*}{Definition}

\newtheorem*{example*}{Example}

\newtheorem*{remark*}{Remark}

\newtheorem*{remarks*}{Remarks}

\newtheorem{dut}[lemma]{Discrete Uniformization Theorem}

\newtheorem{tarski}[lemma]{Tarski's Theorem}

\newtheorem{maintheorem}{Theorem}

\newtheorem*{mainindex*}{Main Index Theorem (weak form)}

\theoremstyle{definition}

\newtheorem*{defi*}{Definition}			\newtheorem*{bei*}{Example}
\newtheorem*{sat*}{Theorem}				\newtheorem*{kor*}{Corollary}
\newtheorem*{rmk*}{Remark}					

\let\ssection=\section
\renewcommand{\section}{\setcounter{equation}{0}\ssection}

\newtheorem*{namedtheorem}{\theoremname}
\newcommand{\theoremname}{testing}

\theoremstyle{remark}

\newcommand{\BC}{\mathbb C}			\newcommand{\BH}{\mathbb H}
			
			\newcommand{\BQ}{\mathbb Q}

\DeclareMathOperator{\PSL}{PSL}		
\DeclareMathOperator{\Isom}{Isom}	
\DeclareMathOperator{\Hom}{Hom}		

\DeclareMathOperator{\Gal}{Gal}
\renewcommand{\Im}{\operatorname{Im}}
\renewcommand{\Re}{\operatorname{Re}}

\begin{document}

\title{Three techniques for obtaining algebraic circle packings}
\author{Larsen Louder,\\Andrey Mishchenko,\\Juan Souto}
\date{\today}
\thanks{The second author has been partially supported by NSF grants DMS-0456940, DMS-0555750, DMS-0801029, DMS-1101373.  The third author has been partially supported by NSF grant DMS-0706878 and the Alfred P. Sloan Foundation.  MSC2010 subject classification: primary 52C26, secondary 12L12, 03C60}
\begin{abstract}
The main purpose of this article is to demonstrate three techniques for proving algebraicity statements about circle packings.  We give proofs of three related theorems: (1) that every finite simple planar graph is the contact graph of a circle packing on $\hat\bbC$, equivalently in $\bbC$, all of whose tangency points, centers, and radii are algebraic, (2) that every flat conformal torus which admits a circle packing whose contact graph triangulates the torus has algebraic modulus, and (3) that if $R$ is a compact Riemann surface of genus at least 2, having constant curvature $-1$, which admits a circle packing whose contact graph triangulates $R$, then $R$ is isomorphic to the quotient of $\bbH^2$ by a subgroup of $\PSL_2(\bbR \cap \bar\bbQ)$.  The statement (1) is original, while (2) and (3) have been previously proved in \cite{mccaughan-thesis}*{Chapters 8, 9} the Ph.D.\ thesis of McCaughan.

Our first proof technique is to apply Tarski's Theorem, a result from model theory, which says that if an elementary statement in the theory of real-closed fields is true over one real-closed field, then it is true over any real closed field.  This technique works to prove (1) and (2).  Our second proof technique is via an algebraicity result of Thurston on finite co-volume discrete subgroups of $\PSL_2\bbC\subset \Isom\bbH^3$.  This technique works to prove (1).  Our first and second techniques had not previously been applied in this area.  Our third and final technique is via a lemma in real algebraic geometry, and was previously used by McCaughan to prove (2) and (3).  We show that in fact it may be used to prove (1) as well.
\end{abstract}
\maketitle

\tableofcontents

\section{Introduction}

A \emph{circle packing}  in $\hat\bbC$ is defined to be a finite collection of pairwise interiorwise disjoint metric closed disks in the Riemann sphere $\hat\bbC=\bbC\cup\{\infty\}$ equipped with the constant curvature $+1$ metric as usual.  There are no conditions on the radii of the disks.  We say that two closed disks are \emph{tangent} if their boundary circles are tangent.  The \emph{contact graph} of a circle packing $\calP$ is the graph $G$ whose vertex set is in bijection with the disks of $\calP$, so that two vertices share an edge if and only if the corresponding disks are tangent.

Recall that a graph is called \emph{planar} if it can be homeomorphically embedded in the plane.  The contact graph $G$ of a circle packing in $\hat\bbC$ is always planar, because we may draw the vertices of $G$ at the metric centers of the disks they correspond to, and connect adjacent vertices with geodesic arcs.  It turns out that conversely, any planar graph $G$ can be realized by a circle packing $\calP$ in $\hat\bbC$.  Furthermore, if $G$ is the 1-skeleton of a triangulation of the 2-sphere $\bbS^2$, then this $\calP$ unique, up to action by M\"obius and anti-M\"obius transformations on $\hat\bbC$.  These results together constitute what is probably the most important theorem in circle packing, known as the Koebe \cite{koebe-1936}--Andreev \cite{MR0273510}--Thurston\footnote{Originally presented at his talk at the International Congress of Mathematicians, Helsinki, 1978, according to \cite{MR1303402}*{p.\ 135}.  See also \cite{thurston-gt3m-notes}*{Chapter 13}.} theorem.  A major motivating factor for the study of circle packings is that they are, in a certain precise sense, a discrete approximation to the Riemann mapping, as originally conjectured by Thurston\footnote{In his address at the International Symposium in Celebration of the Proof of the Bieberbach Conjecture, Purdue University, March 1985, according to \cite{MR1207210}*{p.\ 371}.} and proved by Rodin and Sullivan in \cite{MR906396}.  The close relationship between circle packing and e.g.\ classical complex analysis has been further confirmed many times, for example see \cite{MR1207210}.  For further exposition and references on circle packing, see for example the articles \citelist{\cite{MR1303402} \cite{MR2884870}} and their bibliographies, or the book \cite{MR2131318} by Stephenson.\medskip

In this article we present some tools for uniformizing circle packings to make the centers and radii of their disks algebraic.  Precisely, a circle packing $\calP$ in $\hat\bbC$ is called \emph{algebraic} if the following are all algebraic:

\begin{itemize}
\item all tangency points between pairs of disks in $\calP$,
\item all centers of disks in $\calP$, and
\item all Euclidean radii of disks in $\calP$ under stereographic projection.
\end{itemize}

\noindent There are two equivalent definitions of algebraicity of a point $z\in \hat\bbC$.  First, we may say that $z\in \hat\bbC = \bbC \cup \{\infty\}$ is algebraic if and only if either the image of $z$ under standard stereographic projection to $\bbC$ is algebraic, or $z=\infty$.  Second, we may say that $z\in \hat\bbC = \bbP^1(\bbC)$ is algebraic if and only if $z\in \bbP^1(\bar\bbQ) \subset \bbP^1(\bbC)$.

The first main theorem of this article is:

\begin{maintheorem}
\label{main1}
\label{every g admits an algebraic packing}
Let $G$ be a finite, simple, planar graph.  Then there exists an algebraic circle packing in $\hat\bbC$ having contact graph $G$.
\end{maintheorem}

\noindent Recall that a graph is called \emph{simple} if it is undirected, does not have loops, and has no repeated edges.  A \emph{loop} is an edge from a vertex to itself.  To the best of our knowledge, Theorem \ref{every g admits an algebraic packing} has never appeared in print.\bigskip

Next, we consider circle packings on more general Riemann surfaces.  A \emph{circle packing} on a Riemann surface $R$ is a collection of pairwise interiorwise disjoint metric closed disks on $R$.  The following is well-known, see \cite{MR1207210}:

\begin{dut}
\label{dut}
Let $X$ be a triangulation of a compact oriented topological surface without boundary.  Then there exists a constant curvature closed Riemann surface $R$ homeomorphic to $X$, and a locally finite circle packing $\calP$ on $R$, so that $\calP$ realizes $X$.  Furthermore $R$ is uniquely determined by $X$ up to action by conformal isomorphisms on $R$, and $\calP$ is then uniquely determined by $X$ and $R$ up to action by conformal and anti-conformal automorphisms of $R$.
\end{dut}

\noindent A \emph{triangulation}\label{triangulations} $X = (V,E,F)$ of a surface $S$ is a collection of triangular faces $F$, having vertices $V$ and edges $E$, with identifications along the edges so that the resulting object is homeomorphic to $S$.  We require that two distinct faces of $X$ meet along a single edge, or at a single vertex, or not at all, and that there are no vertex nor edge identifications along the boundary of any one triangle.  Then the 1-skeleton of $X$ is the graph $(V,E)$.  We consider triangulations only up to their combinatorics.  We say that the circle packing $\calP$ \emph{realizes} $X$ if the contact graph of $\calP$ is equal to the 1-skeleton of $X$.\medskip

Call a Riemann surface $R$ which comes from some triangulation $X$ via Theorem \ref{dut} \emph{circle packable}.  In other words, we say that $R$ is \emph{circle packable} if there exists a locally finite circle packing $\calP$ on $R$ realizing some triangulation of $R$.  Then Theorem \ref{dut} implies that there are many $R$ which fail to be circle packable: this is because for any $g\ge 1$ the moduli space of genus $g$ Riemann surfaces is uncountable, but there are only countably many triangulations, considered up to combinatorial equivalence, of a genus $g$ surface.  It is then natural to ask which compact Riemann surfaces are circle packable.

In \cite{MR860677}, Brooks proves that the circle packable compact Riemann surfaces $R$ of given genus are dense in their moduli space.  The next two theorems we state here, originally due to \cite{mccaughan-thesis}*{Chapters 8, 9}, give further information in the direction of an answer to this question:

\begin{maintheorem}
\label{main2}
\label{circle packable tori have algebraic modulus}
Suppose that $\bbT$ is a flat circle packable torus.  Then $\bbT$ is similar or isometric to $\bbC / \left<1,\tau\right>$, with $\tau$ some algebraic number in the upper half-plane.
\end{maintheorem}

\noindent In other words, any flat circle packable torus has algebraic modulus.  We also show that if $\calP$ is a circle packing in $\bbC / \left<1,\tau\right>$ with $\tau$ in the upper half-plane, then the disks of $\calP$ have algebraic radii, and furthermore $\calP$ may be chosen so that all of its disks' centers and tangency points lift to algebraic numbers in $\bbC$.  Next:

\begin{maintheorem}
\label{main3}
\label{circle packable compact surfaces are algebraic}
Suppose that $R$ is a compact circle packable Riemann surface without boundary of genus at least 2, and of constant curvature $-1$.  Then $R$ is isometric to the quotient of $\bbH^2$ by a subgroup of $\PSL_2(\bbR \cap \bar \bbQ)\subset \PSL_2\bbR = \Isom^+\bbH^2$.
\end{maintheorem}

\noindent Here $\bbH^2$ denotes the hyperbolic plane, and $\Isom^+\bbH^2$ the group of orientation-preserving isometries of $\bbH^2$.  We also show that then we may take the (hyperbolic) centers of the circles to be algebraic, and also that each circle has a hyperbolic radius which is the logarithm of an algebraic number.\medskip

Beyond these results, and the denseness theorem mentioned above, not much is known about which compact Riemann surfaces are circle packable.  The analogous question in the open surface case has been answered completely by Williams in \cite{MR2001072}, where he proves that every open complete constant curvature Riemann surface $R$ admits a locally finite circle packing having contact graph triangulating $R$.  More recently, Kojima, Mizushima, Tan, and others have studied the closely related question of which projective structures on surfaces are packable.  An excellent recent survey of results and open problems in that area is \cite{MR2258757}.  Another closely related question is the following: if $M$ is a complete orientable finite-volume non-compact hyperbolic 3-manifold, then each of its cusps is isometric to $\bbT \times [0,\infty)$, where $\bbT$ is some flat torus.  Then one may ask which flat tori appear in this way.  It is known that only countably many do.  In \cite{MR1316178}, Nimershiem proves that every circle packable flat torus arises as the cusp of some $M$ as above, and gives an explicit construction.\footnote{Thanks to Alan Reid for referring us to the article \cite{MR1316178}.}\bigskip

We now move on to our methods of proof for Theorems \ref{main1}, \ref{main2}, \ref{main3}.  We have three different methods.  Each of the next three sections of this article addresses a different one of these three methods of proof.  Each approach works to prove a subset of Theorems \ref{main1}, \ref{main2}, \ref{main3}, but none seems to work to prove all three.  In each case, we include a discussion of the obstructions to applying a given method to those of our three main theorems which it does not handle.\medskip

The first method, which will prove Theorems \ref{main1} and \ref{main2}, is via a result from model theory:

\begin{tarski}
\label{tarski}
Suppose that $S$ is a first-order sentence in the theory of real-closed fields.  If $S$ is true in one real-closed field, then it is true in every real-closed field.
\end{tarski}

\noindent A \emph{real-closed field} is an ordered field $F$, so that any positive element of $F$ has a square root in $F$, and so that any odd-degree polynomial over $F$ has a root in $F$.  The only two examples we care about here are the real numbers $\bbR$ and the real algebraic numbers $\bbR\cap \bar\bbQ$.  A \emph{first-order sentence in the theory of real-closed fields} is, roughly speaking, a finite logical statement, in which we may use the usual logical connectives and quantifiers ($\Rightarrow$, $\Leftrightarrow$, $\neg$, $\vee$, $\wedge$, $\forall$, $\exists$), as well as the symbols from the theory of real-closed fields ($1$, $0$, $\times$, $+$, $=$, $>$), so that all variables in the statement are appropriately quantified.

Our plan is to translate the statements of Theorems \ref{main1} and \ref{main2} into appropriately constructed first-order sentences which we know to be true over $\bbR$ by other means, and then to apply Tarski's theorem \ref{tarski} to obtain the truth of these sentences over $\bbR\cap \bar\bbQ$.  Theorem \ref{main1} follows almost immediately, but the situation for Theorem \ref{main2} is more subtle.  This proof technique is covered in detail in Section \ref{tarski's section}.  Swan provides a nice expository article on Tarski's theorem \ref{tarski}, including a proof, in \cite{swan-tarski}, and cites the book \cite{MR0219380} of Kreisel and Krivine as his source.  The proof of Tarski's theorem is by elimination of quantifiers.\bigskip

Our second technique works only to prove Theorem \ref{main1}.  It proceeds via the following algebraicity result of Thurston in the flavor of the Mostow--Prasad rigidity theorem, found in \cite{thurston-gt3m-notes}*{Proposition 6.7.4}:

\begin{proposition}
\label{thurston algebraicity prop}
If $\Gamma$ is a discrete subgroup of $\PSL_2\BC$ such that $\BH^3/\Gamma$ has finite volume, then $\Gamma$ is conjugate to a group of matrices whose entries are algebraic.
\end{proposition}

\noindent We sketch a proof of Proposition \ref{thurston algebraicity prop} at the end of Section \ref{thurston's section}.

Our proof of Theorem \ref{main1} proceeds as follows.  We first complete our finite simple planar graph $G$ to the 1-skeleton of a triangulation $X$, and circle pack the resulting triangulation in $\hat\bbC$ by the Koebe--Andreev--Thurston theorem, denoting the resulting packing by $\calP$.  Then every connected component $U$ of $\hat\bbC \setminus \cup_{D\in \calP} D$ is a curvilinear triangle.  For such a $U$, let $D^*_U$ be the metric closed disk in $\hat\bbC$ containing $U$, whose boundary circle passes through the corners of $U$.  Then the collection $\calP^* = \{D^*_U\}$ of all such $D^*_U$ is called the \emph{dual circle packing} to $\calP$.  Let $\hat\Gamma$ be the group generated by reflections in the boundary circles of the disks of $\calP$ and $\calP^*$.  We argue that $\hat\Gamma$ is a discrete, finite covolume subgroup of $\Isom \bbH^3$.  It follows that the group $\Gamma$ consisting of orientation-preserving elements of $\hat\Gamma$ is also discrete and of finite covolume.  Then Proposition \ref{thurston algebraicity prop} allows us to assume without loss of generality that $\Gamma\subset \PSL_2\bar\bbQ$.  It follows easily that the tangency points of pairs of disks in $\calP$ are algebraic.  We obtain the algebraic circle packing whose contact graph is our original graph $G$ as a sub-packing of $\calP$, establishing Theorem \ref{main1}.  We work out the details of this technique in Section \ref{thurston's section}.\bigskip

The approaches via Tarski's theorem \ref{tarski} or Proposition \ref{thurston algebraicity prop} had not previously been used to prove algebraicity statements for circle packings.\bigskip

The third and final technique we exposit is due to McCaughan, and was applied in his Ph.D.\ thesis \cite{mccaughan-thesis}*{Chapters 8, 9} to prove Theorems \ref{main2} and \ref{main3}.  Although it was not noted by McCaughan, this approach also works to prove Theorem \ref{main1}.  The main tool is the following lemma of real algebraic geometry:

\begin{lemma}
\label{mccaughan}
Let $\bbV$ be a real algebraic variety defined over a field $k\subset \bbR$, in $n$ variables, and let $\bf v$ be a point of $\bbV$ which is isolated if we give $\bbV$ the subspace topology from $\bbR^n$.  Then the coordinates of $\bf v$ are algberaic over $k$.
\end{lemma}

\noindent Then the idea is to express the existence of a suitably normalized circle packing in the complex or hyperbolic plane by a set of algebraic equations defining a variety.  One shows that this variety is non-empty and that the point of the variety corresponding to the circle packing is appropriately isolated, and then applies the lemma.  We work out the details of this approach, and also sketch a proof of Lemma \ref{mccaughan}, in Section \ref{mccaughan's section}.\bigskip

\noindent {\bf Acknowledgements.} We thank Ralf Spatzier for pointing us indirectly to Thurston's Proposition \ref{thurston algebraicity prop}.  Thanks to Chris Hall for referring us to Tarski's Theorem \ref{tarski}.  Thanks to Sergiy Merenkov for referring us to McCaughan's thesis \cite{mccaughan-thesis}.  Thanks to Andreas Blass, Dan Hathaway, and Scott Schneider for helpful background references on logic and model theory.

\section{Approach via Tarski's theorem}
\label{tarski's section}
\label{tarski section}

We begin the section with some terminology from model theory, and then go on to prove Theorems \ref{main1} and \ref{main2} using Tarski's theorem \ref{tarski}.  At the end of the section we describe the main difficulty of applying this technique to prove Theorem \ref{main3}.\medskip

We first give the precise definition of \emph{elementary statement}, to go with the statement of Tarski's theorem \ref{tarski}.  First, an \emph{atomic predicate in the theory of real-closed fields} is defined to be a relation of the form $g(x_1,\ldots,x_n) = 0$ or $g(x_1,\ldots,x_n) > 0$, where $g$ is a polynomial with integer coefficients in the variables $x_i$.  We fix once and for all a countable collection of free variables which we may use in our atomic predicates.  The collection of \emph{elementary predicates in the theory of real-closed fields} is defined to be the smallest collection
\begin{itemize}
\item containing all of the atomic predicates,
\item closed under negation $\neg$, disjunction $\vee$, and conjunction $\wedge$,
\item closed under existential quantification $(\exists x) P(x, y_1,\ldots,y_n)$, and
\item closed under universal quantification $(\forall x) P(x, y_1,\ldots,y_n)$.
\end{itemize}
In particular, elementary predicates are finitely long.  Finally, an \emph{elementary statement in the theory of real-closed fields} is an elementary predicate in that theory having no free variables, that is, all of whose variables have been appropriately quantified.

Tarski's theorem \ref{tarski} says that one cannot distinguish between different real-closed fields using only elementary statements.  We consider a simpler example: an \emph{elementary statement in the theory of groups} is defined in the analogous way, but the atomic predicates are just relations of the form $x_1\times x_2\times \cdots \times x_n = 1$.  We \emph{can} distinguish between different groups in this restricted language: for example, the statement $\exists x: \neg(x = 1)\wedge (x \times x \times x = 1)$ is true in $\bbZ/3\bbZ$, but is false in $\bbZ$.  Tarski's theorem says that this does not happen in the theory of real-closed fields.  For a more thorough exposition on these definitions and on Tarski's theorem, see \cite{swan-tarski}*{Section 2}.\medskip

\begin{proof}[Proof of Theorem \ref{main1}]
Let $G$ be a finite simple planar graph.  The existence of a circle packing having contact graph $G$ may be expressed as follows.  Enumerate the vertices of $V$ by $v_1,\ldots,v_n$.  For every vertex $v_i$ of $V$, let $x_i, y_i, r_i$ be free variables, which will naturally represent the $x$-coordinate, $y$-coordinate, and Euclidean radius, of the disk corresponding to $v_i$.
\begin{itemize}
\item For $1\le i\le n$, let $R_i$ be the atomic predicate $r_i > 0$.  These statement will ensure that the radii of our disks are positive.
\item Enumerate the edges $e_1,\ldots,e_m$ of $G$.  For every $e_i = \left<v_j, v_k\right>$, let $E_i$ be the atomic predicate $(x_j - x_k)^2 + (y_j - y_k)^2 = (r_j+r_k)^2$.  This statement will ensure that the disks corresponding to $v_j$ and $v_k$ are tangent.
\item Enumerate the pairs of vertices $v_j, v_k$ of $V$ which do \emph{not} share an edge.  Suppose there are $\ell$ such pairs.  For every such pair, let $F_i$, with $1\le i\le \ell$, be the atomic predicate $(x_j-x_k)^2 + (y_j-y_k)^2 > (r_j+r_k)^2$.  This statement will ensure that the disks corresponding to $v_j$ and $v_k$ do not meet.
\item Then, let $S$ be the elementary statement:
\[
\exists x_1,\ldots, x_n, y_1,\ldots, y_n, r_1,\ldots,r_n: R_1\wedge \cdots \wedge R_n \wedge   E_1 \wedge \cdots \wedge E_m \wedge  F_1 \wedge \cdots \wedge F_\ell 
\]
\end{itemize}

\noindent Then the truth of $S$ if the existential symbol is taken to quantify over $\bbR$ is equivalent to the existence of a circle packing in $\bbC$ having contact graph $G$.  Furthermore $S$ is true when quantified over $\bbR$, by the Koebe--Andreev--Thurston theorem, as we discussed before.  Thus by Tarski's theorem \ref{tarski}, the statement $S$ is also true when quantified over $\bbR \cap \bar\bbQ$.  Theorem \ref{main1} follows immediately.
\end{proof}

\begin{proof}[Proof of Theorem \ref{main2}]
Our approach is along the same lines as in the proof of Theorem \ref{main1}, but there are added difficulties.  We begin by pointing out the complications, and outlining our plan.  Let $X$ be a triangulation of a topological 2-torus.  Note first that the existence of a circle packing on some flat torus $\bbT$ realizing $X$ is equivalent to the existence of a doubly translation-periodic circle packing in the universal cover $\bbC$ of $\bbT$, realizing the lift of the triangulation $X$.  Thus we wish to express, via an elementary statement in the theory of real-closed fields, the existence of a $\tau$ in the upper half-plane, and a $\left<1,\tau\right>$-periodic circle packing $\tilde \calP$ in $\bbC$, so that the image of $\tilde\calP$ in the quotient $\bbT = \bbC/\left<1,\tau\right>$ is a circle packing $\calP$ in the torus $\bbT$ having contact graph equal to the 1-skeleton of $X$.  However, our elementary statement cannot have an infinite number of constituent atomic predicates, so we cannot simply pick $x_i, y_i, r_i$ for every vertex of the infinite packing $\tilde\calP$ and quantify over them.  Instead, our plan is to pick a sort of ``fundamental domain'' for $\tilde\calP$, having only finitely many disks, so that the image of these disks under the action by $\left<1,\tau\right>$ is all of $\tilde\calP$.\medskip

We now work out the proof in detail.  Beginning with the triangulation $X$, let $\bbT$ be a flat torus, and let $\calP$ be the circle packing in $\bbT$ having contact graph equal to the 1-skeleton of $X$, by the Discrete Uniformization Theorem \ref{dut}.  Let $\tilde\calP$ be its lift to $\bbC$.  We now apply some normalizations to $\tilde\calP$.

First, note that for example $\left<1,\tau\right>$ and $\left<1,\tau+1\right>$ act the same way on $\bbC$, so the choice of $\tau$ in the generators of the group which quotients $\tilde\calP$ to $\calP$ is not unique.  However, it is well-known (see \cite{MR510197}*{Section 7.2.3, Theorem 2}) that every flat torus is equivalent (by a similarity) to exactly one $\bbC/\left<1,\tau\right>$ for $\tau$ in the so-called \emph{fundamental region of the modular group}, which we will denote by $V$.  Specifically, letting $H = \{ z\in \bbC : \Im(z) > 0 \}$ denote the upper half-plane, we have:
\begin{align*}
V =  \left\{ z \in H: |z| > 1,\, |\Re(z)| < \frac{1}{2} \right\} & \cup \left\{ z \in H: |z| \geq 1,\, \Re(z)= - \frac{1}{2} \right\} \\
& \cup \left\{ z \in H: \left| z \right| = 1,\, -\frac{1}{2} < \Re(z)\leq 0 \right\}
\end{align*}
Normalize $\tilde\calP$ by Euclidean similarities so that $\calP = \tilde\calP / \left<1,\tau\right>$ for some $\tau\in V$.  By the essential uniqueness of $R$ in the Discrete Uniformization Theorem \ref{dut}, we have precisely one choice of $\tau \in V$ for this normalization.

Next, pick once and for all a distinguished vertex $v_0$ of $X$.  Let $D_0$ denote the disk of $\calP$ corresponding to $v_0$.  Apply translations to $\tilde\calP$ so that some lift $\tilde D_0$ of $D_0\in \calP$ is centered at the origin.  Then in particular the lifts of $D_0$ in $\tilde\calP$ are exactly the images of $\tilde D_0$ under the action by $\left<1,\tau\right>$.  Furthermore, by the essential uniqueness of $\calP$ in the Discrete Uniformization Theorem \ref{dut}, the resulting packing $\tilde\calP$ is the unique one which is $\left<1,\tau\right>$-periodic, so that there is a lift of $D_0$ centered at the origin.

(Actually, the last sentence is only true if $\bbC / \left<1,\tau\right>$ has no non-trivial automorphisms.  If $\bbC / \left<1,\tau\right>$ has non-trivial automorphisms then we will require one more normalizing condition to identify $\tilde\calP$ uniquely, but almost any reasonably chosen one will do, and it will be clear how to adapt the rest of our argument to include the extra normalization.  We therefore make no further mention of this point.)

We fix a fundamental domain in $\bbC$ for the action by $\left<1,\tau\right>$.  Let $P$ be the parallelogram spanned by $1$ and $\tau$, so that the following hold:
\begin{itemize}
\item The only corner we include in $P$ is the bottom-left-hand corner, the origin.
\item We include the open bottom and left sides, but neither the top nor the right side of the parallelogram $P$.
\end{itemize}
Explicitly, if $\tau = a + b \sqrt{-1}$ for $a,b\in \bbR$, then we may express $P$ as:
\[
P = \{ z \in \bbC : 0 \le \Im(z) < b, (b/a) \Im(z) \le \Re(z) < (b/a) \Im(z) + 1 \}
\phantomsection
\label{fundamental parallelogram}
\]
This expression will be helpful later.

Let $V_P$ be the set of vertices of the contact graph of $\tilde\calP$ which correspond to disks whose centers lie in $P$.  The notation in the following discussion can get messy, so to help we adopt the convention that if $v$ is a vertex of the contact graph of a packing, say $\calP$, then we write $\calP(v)$ to denote the disk to which $v$ corresponds.  Similarly, if $S$ is a subset of vertices of the contact graph of $\calP$, then $\calP(S)$ is the collection of disks of $\calP$ corresponding to vertices of $S$.  Then for example the image of $\tilde\calP(V_P)$ under the action by $\left<1,\tau\right>$ is all of $\tilde\calP$, and furthermore no two disks of $\tilde\calP(V_P)$ are identified under this action.

Let $V_{P+1}$ be the set of vertices of the contact graph of $\tilde\calP$ which correspond to disks whose centers lie in $P+1$, similarly $V_{P+\tau}, V_{P+1+\tau}$.  Note that then $V_P, V_{P+1}, V_{P+\tau}, V_{P+1+\tau}$ are pairwise disjoint.  Let $\tilde\calQ$ be the sub-packing of $\tilde\calP$ consisting only of those disks corresponding to vertices in $V_P \cup V_{P+1} \cup V_{P+\tau} \cup V_{P+1+\tau}$.  This $\tilde\calQ$ will act as our ``fundamental domain'' for the packing $\tilde\calP$.
\medskip

We now construct an elementary statement in the same style as in the proof of Theorem \ref{main1} above.  Enumerate the vertices $ v_1,\ldots, v_{4n}$ of the contact graph $G(\tilde\calQ)$ of $\tilde\calQ$.  Label the $v_i$ so that
\begin{itemize}
\item the disks of $\tilde\calQ$ having centers in $P$ are exactly $\tilde\calQ(v_1),\ldots,\tilde\calQ(v_n)$,
\item those having centers in $P + 1$ are $\tilde\calQ(v_{n+1}),\ldots,\tilde\calQ(v_{2n})$,
\item those having centers in $P + \tau$ are $\tilde\calQ(v_{2n+1}),\ldots,\tilde\calQ(v_{3n})$, and
\item those having centers in $P + 1+ \tau$ are $\tilde\calQ(v_{3n+1}),\ldots,\tilde\calQ(v_{4n})$.
\end{itemize}
Fix free variables $x_1,\ldots, x_{4n}, y_1,\ldots,y_{4n}, r_1,\ldots,r_{4n}$, which will correspond to the centers and radii of the disks of $\tilde\calQ$ in the natural way.  It will be helpful notationally to fix two additional free variables $a,b$, which will stand for the real and imaginary parts of $\tau$, and three additional free variables $x_0, y_0, r_0$, which will eventually be set equal to $x_i, y_i, r_i$ for some $i$, namely that $i$ for which $\tilde v_i$ corresponds to our distinguished disk $\tilde D_0 \in \tilde\calQ \subset \tilde\calP$.

\begin{itemize}
\item Define $R_i, E_i, F_i$ for $G(\tilde\calQ)$ as we did for $G$ in the proof of Theorem \ref{main1}.
\item Let $v_i$ be the vertex of $G(\tilde\calQ)$ corresponding to the disk $\tilde D_0 \in \tilde\calQ \subset \tilde\calP$.  Let $Z$ be the elementary predicate $x_i = x_0\wedge y_i = y_0 \wedge r_i = r_0 \wedge x_0 = 0 \wedge y_0 = 0$.  We do this for notational convenience, and to encode the normalization that $\tilde D_0$ is centered at the origin.
\item Recall that the disks $\tilde D_0 + \tau$ and $\tilde D_0 + 1$ both lie in $\tilde\calQ$.  Suppose that the vertices of $G(\tilde\calQ)$ corresponding to these two translates of $\tilde D_0$ are $v_j$ and $v_k$ respectively.  Then let $T$ be the elementary predicate $x_j = a \wedge y_j = b \wedge x_k = 1 \wedge y_k = 0$.  This statement further encodes the normalizations on our packing, and picks out $\tau = a + b\sqrt{-1}$ for us.  (We write $\sqrt{-1}$ to avoid overloading the variable $i$, which we use frequently as an indexing variable.)
\item Let $W$ be the elementary predicate expressing that $\tau = a + b\sqrt{-1}$ lies in the fundamental region $V$ of the modular group, in terms of only $a$ and $b$.  It is clear how to do this given the equation defining $V$, although to write it out would be messy and unilluminating, so we leave it to the dedicated reader.
\item Enumerate the pairs of disks of $\tilde\calQ$ which are identified under the action of $\left<1,\tau\right>$.  Suppose there are $p$ such pairs.  For the $i$th such pair $\tilde\calQ(v_j), \tilde\calQ(v_k)$, with $1\le i\le p$, we have that $\tilde\calQ(v_j) = \tilde\calQ(v_k) + s + t\tau$, for some pair of integers $s$ and $t$.  Then let $U_i$ be the elementary predicate $r_j = r_k \wedge x_j = x_k + s + ta \wedge y_j = y_k + tb$.  These statements encode the desired action of $\left<1,\tau\right>$ on our packing.  Note that $s$ and $t$ are not variables in $U_i$, but are actual integers.
\item For $1\le i\le n$, let $K_i$ be the elementary predicate expressing that $(x_i, y_i) \in P$, in terms of only $x_i, y_i, a, b$.  Again, it is clear how to do this, given the equation defining $P$.  Similarly,
\begin{itemize}
\item for $n+1 \le i \le 2n$, let $K_i$ express that $(x_i,y_i) \in P+1$,
\item for $2n+1 \le i\le 3n$, let $K_i$ express that $(x_i,y_i)\in P + a + b\sqrt{-1}$, and
\item for $3n+1 \le i\le 4n$, let $K_i$ express that $(x_i,y_i)\in P + 1 +  a + b\sqrt{-1}$.
\end{itemize}
\item Then, let $S$ be the elementary statement:
\begin{align*}
\exists x_1,\ldots, x_{3n}, y_1,\ldots, y_{3n}, r_1,\ldots,r_{3n}: & R_1\wedge \cdots \wedge R_{3n} \wedge E_1 \wedge \cdots \wedge E_m \\
\wedge &  F_1 \wedge \cdots \wedge F_\ell \wedge U_1 \wedge \cdots \wedge U_p \\
\wedge & K_1\wedge \cdots \wedge K_{3n} \wedge Z \wedge T \wedge W
\end{align*}
\end{itemize}
Clearly $S$ is true when quantified over $\bbR$, because the centers and radii of the disks of the packing $\tilde\calQ$ satisfy $S$.  Then by Tarski's theorem \ref{tarski} the statement $S$ is true when quantified over $\bbR \cap \bar\bbQ$ as well.  We wish to show that there is a unique tuple of values for the $x_i, y_i, r_i \in \bbR$ making the constituent predicates of $S$ true, and it will then follow that in fact these must lie in $\bbR \cap \bar \bbQ$, completing the proof.\medskip

Suppose that $x_1,\ldots,x_{3n},y_1,\ldots,y_{3n},r_1,\ldots,r_{3n}, x_0, y_0, r_0, a, b\in \bbR$ is any solution to the statement $S$.  That is, if we plug in these values for the variables into $S$, then all of its constituent predicates evaluate to true.  Let $\tilde\calD_0$ be the circle packing in $\bbC$ having $n$ disks, so that the $i$th disk of $\tilde\calD_0$ has center $(x_i,y_i)$ and radius $r_i$.  Now, the disks of $\tilde\calD_0 + 1$ are exactly those whose centers and radii are $(x_i, y_i), r_i$ for $n+1 \le i\le 2n$, so the disks of $\tilde\calD_0$ and $\tilde\calD_0 + 1$ are pairwise interiorwise disjoint.  Similarly, the disks of $\tilde\calD_0$ are pairwise interiorwise disjoint from those of $\tilde\calD_0 + a + b \sqrt{-1}$ and of $\tilde\calD_0 + 1 + a + b\sqrt{-1}$.  It follows that $\tilde\calD = \cup_{\lambda \in \left<1,a + b\sqrt{-1}\right>} \tilde\calD_0 + \lambda$ is a circle packing in $\bbC$, which is invariant under the action of $\left<1,a + b\sqrt{-1}\right>$.  Furthermore it is clear that $\calD = \tilde\calD / \left<1,a + b\sqrt{-1}\right>$ has the same contact graph as our initial packing $\calP$.  Finally, by construction of $S$, we have that $\tilde\calD$ is normalized so that some lift of $\calD(v_0)$ is centered at the origin, with $a + b\sqrt{-1} \in V$.  It follows that $\tilde\calD = \tilde\calP$, and Theorem \ref{main2} is proved.
\end{proof}

We conclude the section with a brief discussion of the main difficulty in applying this approach to  prove Theorem \ref{main3}:\medskip

The natural approach is to try to follow the proof of Theorem \ref{main2}.  Then the $x_i, y_i, r_i$ can be taken to represent hyperbolic centers and radii in the Poincar\'e disk model of hyperbolic space.  It turns out that it is possible to express algebraically that the hyperbolic distance between $(x_i,y_i)$ and $(x_j,y_j)$ is $r_i + r_j$, see our proof of Theorem \ref{main3} in Section \ref{mccaughan's section}, so this is not an obstruction.  It is also possible to algebraically normalize the lift $\tilde\calP$ of our packing $\calP$ in $R$, where $R$ is a compact complete constant curvature $-1$ Riemann surface.  Suppose that $R = \bbH^2 / \Gamma$ for $\Gamma\subset \PSL_2\bbR = \Isom^+ \bbH^2$, so that $\tilde\calP /\Gamma = \calP$.  Then the elements of $\Gamma$ are matrices over $\bbR$, and it is possible to express the action of such a matrix on a point $z\in \bbH^2 \cong \bbD$ algebraically in terms of its matrix entries and the coordinates of $z$.

The difficulty in making the proof go through is in picking out generators for the quotient group $\Gamma$ so that $\bbH^2 / \Gamma \cong R$.  In our proof of Theorem \ref{main2} given immediately above, we were fortunate to be able to pick these generators to be $1$ and $\tau\in V$.  In our setting, if we wish to write down an elementary sentence $S$ which is to have precisely one solution over $\bbR$, expressing the existence of $\tilde\calP$ and our selected generators for $\Gamma$, then we need a normalization on our generating set which identifies it uniquely.  This normalization needs to be expressible via a finite list of elementary predicates in the variables we quantify over in $S$.  It is not clear how to do this.

The proof in Section \ref{mccaughan's section} follows a similar strategy to the one we have described here.  However, the reason that the proof in Section \ref{mccaughan's section} goes through is that it expresses the existence of the desired normalized circle packing $\tilde\calP$ via polynomial equations defining an algebraic variety, and in that setting we may encode all of the infinitely many contacts between the disks of $\tilde\calP$.  This is not possible in the present setting, as elementary sentences must be finite.

\section{Approach via hyperbolic geometry and Proposition \ref{thurston algebraicity prop}}
\label{thurston's section}

We begin the section with a proof of Theorem \ref{main1}, using Proposition \ref{thurston algebraicity prop}.  This section relies on some basic, standard tools and facts from hyperbolic geometry.  For a general reference on hyperbolic geometry, see for example \cite{MR2249478}.  After the proof of Theorem \ref{main1}, we sketch the proof of Proposition \ref{thurston algebraicity prop}, for the convenience of the reader.  Finally, we briefly remark on the potential difficulty of applying this approach to Theorems \ref{main2} and \ref{main3}.

\begin{proof}[Proof of Theorem \ref{main1}]
Let $G$ be a finite simple planar graph.  We show that there is a circle packing $\calD$ in $\bbC$ having contact graph $G$, so that all Euclidean centers and radii of disks of $\calD$, as well as all tangency points of pairs of disks of $\calD$, are algebraic numbers.

First, let $X$ be a triangulation of $\bbS^2$, so that the 1-skeleton of $X$ has $G$ as a subgraph.  We will find a circle packing in $\bbC$ realizing $X$, which has the algebraicity properties described in the previous paragraph, and obtain $\calD$ as a sub-packing, completing the proof.

We recall a notational convention from the previous section, that if $v$ is a vertex of the contact graph of the circle packing $\calP$, then $\calP(v)$ denotes the disk of $\calP$ corresponding to the vertex $v$.

Let $\calP$ be any circle packing in $\hat\bbC$ realizing $X$.  Then all of the connected components of $\hat\bbC \setminus \cup_{D\in \calP} D$ are curvilinear triangles, and furthermore these triangles are in natural bijection with the set of faces $F$ of $X$.  We now construct the so-called \emph{dual circle packing} to $\calP$, denoted $\calP^*$.  For each such triangle $T_f$, corresponding to the face $f = \left<v_1,v_2,v_3\right>\in F$, let $\calP^*(f)$ denote the closed disk containing $T_f$, whose boundary circle passes through the three corners of $T_f$, that is, the three tangency points between pairs of the disks $\calP(v_1), \calP(v_2), \calP(v_3)$.  Then $\calP^*$ is defined to be the collection of disks $\{\calP^*(f)\}_{f\in F}$.  It is well-known that the boundary circle $\partial\calP^*(f)$ is orthogonal to all three of $\partial \calP(v_1), \partial\calP(v_1), \partial\calP(v_2)$, see Figure \ref{orthogonality}.  It follows that the disks of $\calP^*$ are pairwise interiorwise disjoint, thus that $\calP^*$ is a circle packing in $\hat\bbC$.  Furthermore the two packings $\calP$ and $\calP^*$ together completely cover $\hat\bbC$.

\begin{figure}
\centering
\scalebox{1} 
{
\begin{pspicture}(0,-2.55)(5.0,2.55)
\pscircle[linewidth=0.018,dimen=outer](1.35,1.2){1.35}
\pscircle[linewidth=0.018,dimen=outer](3.84,1.19){1.16}
\pscircle[linewidth=0.018,dimen=outer](2.79,-1.14){1.41}
\pscircle[linewidth=0.018,linestyle=dashed,dash=0.16cm 0.16cm,dimen=outer](2.71,0.44){0.77}
\psdots[dotsize=0.12](2.7,1.19)
\psdots[dotsize=0.12](2.06,0.05)
\psdots[dotsize=0.12](3.38,0.13)
\usefont{T1}{ptm}{m}{n}
\rput(2.9814062,1.415){$\theta_1$}
\usefont{T1}{ptm}{m}{n}
\rput(2.4214063,0.875){$\theta_2$}
\usefont{T1}{ptm}{m}{n}
\rput(2.1414063,0.415){$\theta_3$}
\usefont{T1}{ptm}{m}{n}
\rput(2.0814064,-0.265){$\theta_4$}
\usefont{T1}{ptm}{m}{n}
\rput(3.4614062,-0.205){$\theta_5$}
\usefont{T1}{ptm}{m}{n}
\rput(3.2614062,0.495){$\theta_6$}
\usefont{T1}{ptm}{m}{n}
\rput(2.9214063,0.915){$\theta_7$}
\end{pspicture} 
}
\caption
{
\label{orthogonality} We get first that $\theta_1 = \theta_2$, $\theta_3 = \theta_4$, $\theta_5 = \theta_6$, because these are pairs of vertical angles.  On the other hand we get that $\theta_2 = \theta_3$, $\theta_4 = \theta_5$, $\theta_6 = \theta_7$, because any pair of round circles which cross meet at the same angle at their two intersection points.  We conclude that $\theta_1 = \theta_7 = \pi/2$.
}
\end{figure}
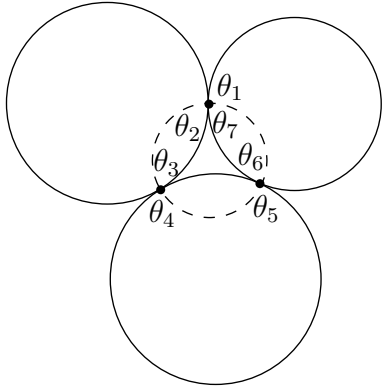

We now wish to define the \emph{reflection group} associated to the packing $\calP$.  We first establish some conventions and notation.  For the rest of this proof, we identify the Riemann sphere $\hat\bbC$ with the boundary at infinity of hyperbolic 3-space $\bbH^3$.  Recall that this identification is consistent with the identification of the group of M\"obius transformations $\PSL_2\bbC$ with the group of orientation-preserving isometries $\Isom^+\bbH^3$ of $\bbH^3$.  Every circle $C$ in $\hat\bbC$ is now the boundary of a totally geodesic copy of $\bbH^2$ in $\bbH^3$.  Denote by $\sigma_C$ the reflection of $\bbH^3$ along this hyperbolic plane.  The induced boundary map of $\sigma_C$ on $\hat\bbC$ is inversion along the circle $C$ in the sense of projective geometry.  Note that $\sigma_C \not \in \Isom^+\bbH^3$, as it is orientation-reversing.

Then the \emph{reflection group} associated to $\calP$, which we will denote by $\hat\Gamma \subset \Isom\bbH^3$, is the group generated by all of the reflections $\sigma_C$, for $C \in \cup_{D\in \calP \cup \calP^*} \partial D$.  The following lemma is well-known, but we sketch its proof for the convenience of the reader.

\begin{lemma}
\label{lattice}
The group $\hat\Gamma$ is discrete and the associated orbifold $\bbH^3/\hat\Gamma$ has finite volume. 
\end{lemma}

\begin{proof}
Given a round closed disk $D$ in $\hat\bbC$, let $H_D$ be the hyperbolic half-space whose closure in $\bbH^3 \cup \hat\bbC = \bbH^3 \cup \partial_\infty \bbH^3$ is equal to $D$.  Let $\Delta$ denote the closure, in $\bbH^3$, of the complement of $\cup_{D\in \calP \cup \calP^*} H_D$.  Then $\Delta$ is a polyhedron in $\bbH^3$ with finitely many faces and $\hat\Gamma$ is the group of isometries of $\bbH^3$ generated by reflections along the faces of $\Delta$.  The fact that the circles of $\calP$ meet those of $\calP^*$ orthogonally implies that the dihedral angles of $\Delta$ are all equal to $\pi/2$.  In particular, it follows from Poincar\'e's polyhedron theorem, see \cite{MR1279064}, that $\hat\Gamma$ is a discrete group with fundamental domain $\Delta$.

We still need to prove that $\Delta$ has finite volume.  The cusps of $\Delta$ are in bijection with the tangency points of pairs of disks of $\calP$, equivalently the tangency points of pairs of disks of $\calP^*$.  Let $p$ be such a tangency point.  Then there are two disks of $\calP$, and two of $\calP^*$, whose boundary circles pass through $p$.  Call these disks $D_1, D_2, D_3, D_4$.  For a moment we consider hyperbolic space $\bbH^3 = \bbC \times \bbR_+$ in the upper half-space model, normalized so that $p = \infty$.  In this model, under this normalization, each of the hyperbolic half-spaces $H_{D_1}, H_{D_2}, H_{D_3}, H_{D_4}$ becomes the intersection of $\bbC \times \bbR_+$ with a vertical Euclidean half-space.  Furthermore because the disks of $\calP$ meet those of $\calP^*$ orthogonally, we in fact have that $\bbH^3 \setminus \cup_{i=1}^4 H_{D_i}$ is an open rectangular column.  A straightforward computation shows that the hyperbolic volume contained in this column above any fixed Euclidean height is finite.  Thus $\Delta$ has finite volume in a neighborhood of any of its cusps.  But it has finitely many cusps, and the complement in $\Delta$ of some small open neighborhoods of its cusps is compact.  It follows that $\Delta$ has finite volume, concluding the proof of Lemma \ref{lattice}.
\end{proof}

Next, let $\Gamma$ denote the group of orientation-preserving elements of $\hat\Gamma$.  Then $\Gamma\subset \PSL_2\bbC \cong \Isom^+\bbC$ is an index 2 subgroup of $\hat\Gamma$.  It follows from Lemma \ref{lattice} that $\Gamma$ is discrete, and that $\bbH^3/\Gamma$ has finite volume.

We can now apply Proposition \ref{thurston algebraicity prop}.  Then there is a $g\in \PSL_2\bbC$ so that $g\Gamma g^{-1} \subset \PSL_2\bar\bbQ$.  Noting that the reflection group associated to the circle packing $g \calP = \{gD\}_{D\in \calP}$ is exactly $g\hat\Gamma g^{-1}$, we conclude that we may assume without loss of generality that $\Gamma \in \PSL_2\bar\bbQ$.  It will also be helpful to assume, again without loss of generality, that $\infty$ lies outside of the disks of $\calP$, so $\calP$ is a packing in $\bbC$.


The condition that a point $z$ is fixed by an element $\gamma\in \PSL_2\bar\bbQ$ is algebraic in $z$ and the matrix entries of $\gamma \in \PSL_2\bar\bbQ$.  Thus we get that any fixed point of $\gamma\in \Gamma$ is algebraic.  On the other hand, if $p$ is a tangency point between some pair of disks $D_1,D_2 \in \calP$, then $p$ is fixed by $\sigma_{\partial D_1}\sigma_{\partial D_2} \in \Gamma$.  We conclude that all tangency points between disks of $\calP$ are algebraic.  It is then a simple computation to show that the Euclidean centers and radii of the disks of $\calP$ are algebraic: this is because there is a unique triple of pairwise tangent disks, whose tangency points are $p_1, p_2, p_3$, and the conditions defining this triple are algebraic in the Euclidean $x$- and $y$-coordinates of their centers, their Euclidean radii, and the coordinates of $p_1, p_2, p_3$.  This concludes the proof of Theorem \ref{main1}.
\end{proof}

Proposition \ref{thurston algebraicity prop} is a particular case of  a more general result of Raghunathan \cite{MR0507234}.  We now sketch the proof of Proposition \ref{thurston algebraicity prop} in the case that $\Gamma$ is cocompact:

\begin{proof}[Proof sketch of Proposition \ref{thurston algebraicity prop}]
Consider $\Gamma$ as an abstract group and denote the embedding of $\Gamma$ inside $\PSL_2\bbC$ by $\rho_0$. We consider $\rho_0$ as an element of $\Hom(\Gamma,\PSL_2 \bbC )$. The group $\PSL_2\BC$ is an affine algebraic group defined over $\BQ$. In particular, $\Hom(\Gamma,\PSL_2\BC)$ is also an affine algebraic variety defined over $\BQ$. Hence, the Galois group $\Gal(\BC/\BQ)$ acts on $\Hom(\Gamma,\PSL_2\BC)$. The variety $\Hom(\Gamma,\PSL_2\BC)$ has only finitely many irreducible components and the action of the Galois group permutes them. It follows that every irreducible component of $\Hom(\Gamma,\PSL_2\BC)$ is invariant by a finite index subgroup of $\Gal(\BC/\BQ)$. This implies that each one of the irreducible components is defined over a finite extension of $\BQ$ and hence defined over $\bar\BQ$. Let $\bbV$ be the irreducible component of $\Hom(\Gamma,\PSL_2\BC)$ containing $\rho_0$. Since $\bbV$ is defined over $\bar\BQ$ and $\bar\BQ$ is algebraically closed, we get that $\bbV$ contains an element $\rho$ which is defined over $\bar\BQ$. It follows from the Mostow rigidity theorem that $\rho$ and $\rho_0$ are conjugated, and this concludes the proof of Proposition \ref{thurston algebraicity prop} in the case that $\Gamma$ is cocompact. In the general (finite covolume) case, one has to replace the variety $\Hom(\Gamma,\PSL_2\BC)$ with the subvariety consisting of representations which map parabolic elements to elements whose trace squared is 4.  The argument then is concluded using Prasad's extension of the Mostow rigidity theorem to the finite volume setting.
\end{proof}

Finally, we discuss the main difficulty in applying this approach to the settings of Theorems \ref{main2} and \ref{main3}.  The natural approach to take is to construct a packing $\calP$ in a compact Riemann surface $R$, lift it to the universal cover (either $\bbC$ or $\bbH^2$) of $R$, obtain algebraicity in the universal cover, and then apply known rigidity results to obtain algebraicity in the quotient.  However, the reflection group of the lifted packing in the universal cover will not have finite covolume in $\bbH^3$, thus Proposition \ref{thurston algebraicity prop} does not apply.  Furthermore the finite covolume hypothesis is essential in the proof of Proposition \ref{thurston algebraicity prop}.

\section{Approach via real algebraic geometry and Lemma \ref{mccaughan}}
\label{mccaughan's section}

We begin by applying Lemma \ref{mccaughan} to prove Theorem \ref{main1}.  This lemma was previously applied by McCaughan in \cite{mccaughan-thesis}*{Chapters 8, 9} to prove Theorems \ref{main2} and \ref{main3}, and we give these proofs below.  He did not observe that the same lemma works to prove Theorem \ref{main1} as well, but there is essentially nothing new in the argument we give.  After these three proofs, we sketch a proof of Lemma \ref{mccaughan} for the convenience of the reader.

\begin{proof}[Proof of Theorem \ref{main1}]
Let $G$ be a finite simple planar graph.  We wish to show the existence of a circle packing $\calD$ in $\bbC$, so that all Euclidean centers and radii of the disks of $\calD$ are algebraic.  It will follow that the points of tangency between disks of $\calD$ are algebraic.

Let $X$ be a triangulation of $\bbS^2$ the 1-skeleton of which has $G$ as a subgraph.  Let $\calP$ be a circle packing in $\bbC$ realizing $X$, by the Koebe--Andreev--Thurston theorem.  We will show that $\calP$ can be taken to have the algebraicity properties described in the preceding paragraph, and we may obtain our desired $\calD$ as a sub-packing of $\calP$, completing the proof of the theorem.

We now apply some normalizations to $\calP$.  First, pick some face $\left<v_1,v_2,v_3\right>$ of $X$.  Normalize $\calP$ so that $\calP(v_1), \calP(v_2), \calP(v_3)$ are disks of Euclidean radius $1$, having Euclidean centers $(-1,0), (1,0), (0,\sqrt{3})$ respectively, and so that the rest of the disks of $\calP$ lie in the bounded curvilinear triangular region formed between $\calP(v_1), \calP(v_2), \calP(v_3)$.  The resulting packing $\calP$ is the unique one having contact graph $X$ and satisfying these normalization conditions.

Suppose that there are $n$ vertices of $X$.  Let $\bbV$ be the algebraic variety in $\bbR^{3n}$ defined by equations of the form $(x_i - x_j)^2 + (y_i - y_j)^2 = (r_i + r_j)^2$ for every pair of vertices $v_i, v_j$ of $X$ sharing an edge, together with the extra equations $r_1 = r_2 = r_3 = 1, x_1 = -1, y_1 = 0, x_2 = 1, y_2 = 0, x_3 = 0, y_3 = \sqrt{3}$.  Let ${\bf v}$ be the $3n$-tuple consisting of the centers and radii of the disks of $\calP$.  Then ${\bf v}\in \bbV$.

We wish to argue that $\bf v$ is isolated in $\bbV$ under the subspace topology from $\bbR^{3n}$.  Then Theorem \ref{main1} will be proved, by Lemma \ref{mccaughan}.  First note that $\bf v$ is well-separated from points of $\bbV$ where some radius $r_i$ is negative or zero, because there are finitely many $r_i$, and all of them are positive in $\bf v$.  Suppose that ${\bf v}_i$ is a sequence of points of $\bbV$, converging to $\bf v$, so that every radius coordinate of every ${\bf v}_i$ is positive.  Pass to a subsequence so that $|{\bf v}_{i+1} - {\bf v}| < |{\bf v}_{i} - {\bf v}|$ for all $i$.  We claim that then eventually ${\bf v}_i = {\bf v}$.

To see why, let $\calP_i$ be the collection of disks having centers and radii given by the coordinates of ${\bf v}_i$ in the natural way.  Note first that if $v_j$ and $v_k$ are vertices of $X$ with no edge between them, then eventually $\calP_i(v_j)$ and $\calP_i(v_k)$ are disjoint, because their limiting disks are disjoint.  On the other hand, if $\left<v_j, v_k\right>$ is an edge of $X$, then the disks $\calP_i(v_j)$ and $\calP_i(v_k)$ are tangent for every $i$, by the construction of $\bbV$.  Thus $\calP_i$ is eventually a circle packing realizing $X$.  Furthermore if $i$ is big enough then the disks of $\calP_i$ are sufficiently close to their partners in $\calP$ that the disks of $\calP_i$ must lie in the bounded curvilinear triangular region formed between $\calP(v_1), \calP(v_2), \calP(v_3)$.  Thus we get that eventually $\calP_i = \calP$ by the essential uniqueness of $\calP$, concluding the proof.
\end{proof}

\begin{proof}[Proof of Theorem \ref{main2}]
Let $X$ be a triangulation of the 2-torus, and let $\bbT$ be a flat torus and $\calP$ a circle packing in $\bbT$ realizing $X$, by the Discrete Uniformization Theorem \ref{dut}.  We wish to show that $\bbT$ is similar to the quotient of $\bbC$ by $\left<1,\tau\right>$ with $\tau$ an algebraic number.  Our proof will also show that then, supposing that $\bbT = \left<1,\tau\right>$, we have that the radii of the disks of $\calP$ are algebraic.  This proof uses ideas similar to those in the proof of Theorem \ref{main2} of Section \ref{tarski section}.

Suppose without loss of generality that $\bbT = \bbC / \left<1,\tau\right>$, and let $\tilde\calP$ be a $\left<1,\tau\right>$-periodic packing in $\bbC$, so that $\tilde\calP / \left<1,\tau\right> = \calP$.  Normalize $\tilde\calP$ so that, say, some distinguished disk $D_0$ of $\calP$ has a lift in $\tilde\calP$ centered at the origin.  In this setting it is not necessary that we apply enough normalizations to $\tilde\calP$ and $\tau$ that they are uniquely determined by our normalizing conditions.

Define a fundamental parallelogram $P\subset \bbC$ for the action by $\left<1,\tau\right>$ as in the proof of Theorem \ref{main2} in Section \ref{tarski section}, see p.\ \pageref{fundamental parallelogram}.  Let $V_P$ be the set of vertices of the contact graph of $\tilde\calP$ corresponding to circles having Euclidean center in $P$.

Enumerate the vertices $v_1,\ldots,v_n$ of $V_P$.  Fix variables $x_1,\ldots,x_n,y_1,\ldots,y_n,r_1,\ldots,r_n$, and two free variables $a$ and $b$, which will serve as the real and imaginary parts of $\tau$.  Let $\bbV\subset \bbR^{3n + 2}$ be the real algebraic variety in these variables defined by equations as follow:
\begin{itemize}
\item Every disk $D$ of $\tilde\calP$ may be written $D = \tilde\calP(v_i) + t + s\tau$, for some $v_i\in V_P$ and integers $s$ and $t$.  (If $D$ has center in $P$ then $t = s = 0$.)  Suppose that the distinct disks $D_1, D_2$ of $\tilde\calP$ meet.  Write $D_1 = \tilde\calP(v_i) + t_1 + s_1\tau$ and $D_2 = \tilde\calP(v_j) + t_2 + s_2\tau$.  Then add $\left[ (x_i + t_1 + s_1 a) - (x_j + t_2 + s_2 a) \right]^2 + \left[ (y_i + s_1 b) - (y_j + s_2 b) \right]^2 = (r_i + r_j)^2$ to the defining equations of $\bbV$.  Do this for every pair of distinct disks of $\tilde\calP$ which meet.
\item To encode our normalization from above, let $v_j\in V_P$ be the vertex of that disk of $\tilde\calP(V_P)$ which is a lift of our distinguished disk $D_0$.  Then add $x_j = 0, y_j = 0$ to the defining equations of $\bbV$.

\end{itemize}

Let $\bf v$ be the $3n+2$-tuple given by the radii, and coordinates of the centers, of the disks of $\tilde\calP(V_P)$, and by the real and imaginary parts of $\tau$, in the natural way.  Then clearly $\bf v\in \bbV$.  Furthermore, we argue, essentially as in the proof of Theorem \ref{main1} given earlier in this section that $\bf v$ is an isolated point of $\bbV$ in the subspace topology from $\bbR^n$, concluding the proof of Theorem \ref{main2} by Lemma \ref{mccaughan}.
\end{proof}

\begin{proof}[Proof of Theorem \ref{main3}]
Let $X$ be a triangulation of a genus $g$ surface for $g\ge 2$, let $R$ be a complete compact constant curvature $-1$ Riemann surface, and let $\calP$ be a circle packing in $R$ realizing $X$.  We wish to show that $R$ is isometric to $\bbH^2 / \Gamma$ for some $\Gamma \subset \PSL_2 (\bbR \cap \bar \bbQ) \subset \PSL_2 \bbR = \Isom^+ \bbH^2$.  Along the way we will also end up showing that the hyperbolic radius of any disk of $\calP$ is the logarithm of an algebraic number.

The proof proceeds essentially as did the proof of Theorem \ref{main2} given in this section.  First, let $\tilde\calP$ be a circle packing in $\bbH^2$, so that the quotient of $\bbH^2$ by some discrete group of isometries of $\bbH^2$ is $R$, and so that the image of $\tilde\calP$ under this quotient map is exactly $\calP$.  Pick once and for all a distinguished vertex $v_0$ of $X$, and a neighbor $u_0$ of $v_0$ in $X$, and normalize $\tilde\calP$ so that some lift $\tilde D_{v_0} \in \tilde\calP$ of $\calP(v_0)$ is centered at the origin, and so that the neighboring lift $\tilde D_{u_0}$ of $\calP(u_0)$ is centered along the positive real axis.  This fixes $\tilde\calP$ uniquely.

Let $\Gamma\subset \PSL_2\bbR$ be the group of isometries of $\bbH^2$ so that $\tilde\calP / \Gamma = \calP$.  Let $\gamma_1,\gamma_1\inv,\ldots,\gamma_g,\gamma_g\inv$ be a set of generators for $\Gamma$.  Fix $P\subset \bbH^2$ to be a fundamental domain for the action by $\Gamma$, so that $P$ contains both the origin, and the hyperbolic center of $\tilde D_{u_0}$.  We ask that no two points of $P$ are identified by the action by $\Gamma$, and that $P$ tiles $\bbH^2$ under the action of $\Gamma$.

Let $V_P$ be the set of vertices of the contact graph of $\tilde\calP$ corresponding to disks whose hyperbolic centers lie in $P$.  Enumerate the vertices $v_1,\ldots,v_n$ of $V_P$, and fix real variables $x_1,\ldots,x_n,y_1,\ldots,y_n,R_1,\ldots,R_n$.  Here the $(x_i,y_i)$ will represent the hyperbolic centers of the disks, in the Poincar\'e disk model of hyperbolic space, and the $R_i = e^{r_i}$ where the $r_i$ represent their hyperbolic radii.  Fix $8g$ additional real variables which will represent the matrix entries of the $2g$ matrices $\gamma_i, \gamma_i\inv$.  Let $\bbV\subset \bbR^{3n + 8g}$ be the real algebraic variety in all of these variables defined by equations as follows:
\begin{itemize}
\item Every disk $D$ of $\tilde\calP$ may be written $D = \alpha \tilde\calP(v_i)$ for some $v_i \in V_P$ and $\alpha \in \Gamma$.  In particular the action of $\gamma$ sends the hyperbolic centers of disks to hyperbolic centers.  Suppose that the distinct disks $D_1, D_2$ of $\tilde\calP$ meet.  We wish to express, via a polynomial equation in our variables from above, that the distance between the hyperbolic centers $z_1, z_2$ of $D_1$ and $D_2$ is exactly $r_i + r_j$.  To do this, recall that if $z_1,z_2 \in \bbD \cong \bbH^2$, then the hyperbolic distance between $z_1$ and $z_2$ may be expressed as $d(z_1,z_2) = \operatorname{arcosh}(1 + \delta(z_1,z_2))$, where:
\begin{align*}
\delta(z_1,z_2) & = 2\frac{|z_1-z_2|^2}{(1-|z_1|^2)(1-|z_2|^2)} \\
\operatorname{arcosh}(x) & = \ln (x + \sqrt{x+1}\sqrt{x-1})
\end{align*}
Therefore the expression $e^{d(z_1,z_2)} = e^{r_i + r_j} = R_iR_j$ may be simplified to a polynomial.  Furthermore $z_1 = \alpha_1(x_i + y_i \sqrt{-1})$, similarly for $z_2$.  Plugging these expressions for $z_1,z_2$ into our formula, we may again simplify to a polynomial expression in the $x_i, y_i, R_i$, and the variables representing the matrix entries of the $\gamma_i, \gamma_i\inv$.  Add the resulting polynomial to the defining equations of $\bbV$ for every such pair of distinct disks of $\tilde\calP$ which meet.
\item As before, encode our normalization on $\tilde\calP$ as follows: let $v_i$ be the vertex of $V_P$ corresponding to $\tilde D_{v_0}$, similarly $v_j$ for $\tilde D_{u_0}$.  Add the polynomials $x_i = y_i = y_j = 0$ to our defining equations for $\bbV$.
\end{itemize}
The conclusion of the proof proceeds as in the proof of Theorem \ref{main2} above.
\end{proof}

\begin{proof}[Proof sketch of Lemma \ref{mccaughan}]
Let $\bbV$ be a variety in $n$ variables defined over a field $k\subset \bbR$, and let ${\bf v} = (v_1,\ldots,v_n)$ be a point of $\bbV$ which is isolated in the subspace topology from $\bbR^n$ on $\bbV$.  We wish to show that then the $v_i$ are algebraic over $k$.  Note that if $\alpha$ is any field homomorphism $k(v_1,\ldots,v_n) \to \bbR$ which fixes $k$, then $\alpha({\bf v}) \in \bbV$.

Suppose for contradiction that some $v_i$ fails to be algebraic over $k$.  Then we may write $k(v_1,\ldots,v_n) = k(y_1,\ldots,y_m,z)$, where the $y_1,\ldots,y_m$ are independent transcendentals over $k$, and $z$ is algebraic over $k(y_1,\ldots,y_m)$.  In particular, every $v_i$ is equal to some rational function in the variables $y_1,\ldots,y_m,z$.

Let $p$ be the minimal polynomial for $z$ over $k(y_1,\ldots,y_m)$.  Then $p$ is irreducible over $k(y_1,\ldots,y_m)$, thus $\partial p/\partial z \ne 0$ at $(y_1,\ldots,y_m,z)$.  Then by the implicit function theorem, in a neighborhood of $(y_1,\ldots,y_m,z)$ we may write the $z$-coordinate of the vanishing set of $p$ as a smooth function of the $y_1,\ldots,y_m$-coordinates.

Therefore, choose $(\hat y_1,\ldots,\hat y_m)$ to be close to $(y_1,\ldots,y_m)$, so that $p(\hat y_1,\ldots,\hat y_m, \hat z) = 0$ for $\hat z$ close to but unequal to $z$.  We may ensure in our choice that the $\hat y_i$ are independent transcendentals over $k$.  Then $\alpha$ fixing $k$ and sending $y_i \mapsto \hat y_i$ and $z \mapsto \hat z$ is a field homomorphism from $k(y_1,\ldots,y_m,z) = k(v_1,\ldots,v_n)$ to $\bbR$, which sends $\bf v$ to another point of $\bbV$, and which can be arranged by our choices earlier in this paragraph to send $\bf v$ as close to itself as desired.  This contradicts the hypothesis that $\bf v$ is isolated in $\bbV$.
\end{proof}

\end{document}